\documentclass[11pt]{article}
\usepackage{appendix}
\usepackage{mathtools}
\usepackage[T1]{fontenc}
\usepackage{amsfonts}
\usepackage{amsmath}
\usepackage{amssymb}
\usepackage{amsthm}
\usepackage{thmtools}
\usepackage{bbm}
\usepackage{bm}
\usepackage{mathrsfs}
\usepackage{color}
\usepackage{pdfsync}

\usepackage{enumitem}
\usepackage{xcolor} 
\definecolor{darkgreen}{rgb}{0,0.5,0} 
\definecolor{darkbkue}{rgb}{1,0,0} 
\usepackage{natbib}
\setcitestyle{authoryear,aysep={,}}  %
\usepackage[colorlinks=true, linkcolor=darkgreen, citecolor=darkgreen, urlcolor=blue]{hyperref}
\makeatletter
\renewcommand\@makefnmark{%
  \hbox{\textsuperscript{\textcolor{darkgreen}{\@thefnmark}}}}
\makeatother
\usepackage{tikz}
\usetikzlibrary{patterns}
\usepackage{mathpazo} %

\DeclareMathOperator*{\argmin}{argmin}

\newcommand{\RR}{\mathbb{R}}
\newcommand{\R}{\RR}

\newcommand{\NN}{\mathbb{N}}

\newcommand{\eps}{ \varepsilon}
\usepackage[margin=25 mm]{geometry}
\newcommand{\mykill}[1]{}
\usepackage[capitalize, nameinlink]{cleveref}

\newlist{hypotheses}{enumerate}{1}
\setlist[hypotheses]{
    label=\textcolor{darkgreen}{(H\arabic*)},
    ref=(H\arabic*)
}

\newlist{assumptions}{enumerate}{1}
\setlist[assumptions]{
    label=\textcolor{darkgreen}{(A\arabic*)},
    ref=(A\arabic*)
}
\theoremstyle{plain}
\newtheorem{theorem}{Theorem}[section]

\newtheorem{lemma}[theorem]{Lemma}

\theoremstyle{definition}\newtheorem{definition}[theorem]{Definition}
\newtheorem{remark}[theorem]{Remark}

\theoremstyle{remark}

\newlist{myenum}{enumerate}{3}
\setlist[myenum,1]{label={\rm (H\arabic*)},
                   ref  ={\rm (H\arabic*)}}
\crefname{myenumi}{property}{properties}

\renewenvironment{thebibliography}[1]{%
\begin{oldthebibliography}{#1}%
\setlength{\baselineskip}{.9em}
\linespread{1}
\small
\setlength{\parskip}{.40ex}%
\setlength{\itemsep}{.30em}%
}%
{%
\end{oldthebibliography}%
}

\definecolor{grape}{rgb}{0.43, 0.17, 0.71}

\begin{document}

\title{ \huge Breakdown properties of optimal transport maps:  general transportation costs }
\date{}
\author{  
  Alberto Gonz{\'a}lez-Sanz
  \and Marco Avella Medina
  \thanks{Department of Statistics, Columbia University, \textcolor{darkgreen}{alberto.gonzalezsanz@columbia.edu}, \textcolor{darkgreen}{marco.avella@columbia.edu}. This research was partly
supported by the  NSF grant DMS-2310973 (Avella Medina)}   
  }
  
\maketitle 
\vspace{-1.5em}
\begin{abstract}
Two recent works, Avella-Medina and González-Sanz (2026) and Passeggeri and Painda\-veine (2026), studied the robustness of the optimal transport map through its breakdown point, i.e., the smallest fraction of contamination that can make the map take arbitrarily aberrant values. Their main finding is the following: let $P$ and $Q$ denote the target and reference measures, respectively, and let $T$ be the optimal transport map for the squared Euclidean cost. 
Then, the breakdown point of $T(u)$, when $P$ is perturbed and $Q$ is fixed, coincides with the Tukey depth of $u$ relative to $Q$. In this  note, we extend this result to general convex cost functions, demonstrating that the cost function does not have any impact on the breakdown point of the optimal transport map.    Our contribution provides a definitive characterization of the breakdown point of the optimal transport map. In particular, it shows that for a broad class of regular cost functions, all transport-based quantiles enjoy the same high breakdown point properties.  
\end{abstract}

 \vspace{1em}

{\small

\noindent \emph{MSC2020 subject classifications. Primary: } 62G35, 62G30.  
}
 \vspace{.2em}

 \section{Introduction}
We consider the  general Monge  optimal transport problem \citep{monge1781,GangboMcCann.96}
\begin{align}\label{eq:Monge}\tag{MP}
\inf_{T:\, T\#Q=P}\int_{\mathbb{R}^d} c({x},T({x})) d Q({x}). 
\end{align}
Here $c:\mathbb{R}^d\times \mathbb{R}^d\to \mathbb{R}_{\geq 0}$ is the cost function and $T\#P$ denotes the \emph{push-forward} measure, that is, the measure such that for each measurable set $A$ we have $T\#P(A):=P(T^{-1}(A))$. We are interested in the solution $T$ of \eqref{eq:Monge}, i.e.~the Monge map. Optimal transport and Monge maps are important tools in applied mathematics \citep{Santambrogio-Book} that in recent years have been used to develop interesting methodology in statistics \citep{HallinDelBarrioCuestaAlbertosMatran.21,chewi:nilesweed:rigollet2025}  and data science \citep{PeyreCuturi.19}.  

The robustness properties of the Monge map are not yet well-established, as highlighted by \citet{ronchetti2023}. Key concepts from robust statistics---such as the influence function for local stability and the breakdown point for global reliability---lack a comprehensive theory in this context. The influence function is a G\^ateaux derivative under Dirac-mass perturbations, but only directional derivatives under certain smooth perturbations have been studied for transport maps \citep{Loeper.09,gonzálezsanz2024linearizationmongeampereequationsdata}. Hence an influence function framework is still absent.

A significant recent development concerns the breakdown point. In parallel works, \citet{paindaveine:passeggeri2024} and \citet{Avella-Gonzalez.25} established a remarkable connection for the squared Euclidean cost: the breakdown point of the Monge map (under perturbations of the target measure \( P \)) evaluated at a point \( u \) is precisely the Tukey depth of \( u \) with respect to the source measure \( Q \). Their proofs, although quite different, hinge on the monotonicity of the Monge map and appear difficult to extend beyond the quadratic cost.

In this note, we generalize this result. We prove that the equivalence between the breakdown point and the Tukey depth holds for costs of the form \( c(x,y) = h(x-y) \), under the standard regularity conditions of \cite{GangboMcCann.96}. %
A particularly interesting consequence of this finding is that  for all such general cost functions, the resulting transport-based quantile estimators enjoy the same high breakdown points. \cite{Avella-Gonzalez.25,paindaveine:passeggeri2024} had shown that the transport-based quantiles built with the quadratic cost, as proposed in \cite{HallinDelBarrioCuestaAlbertosMatran.21, ghosal:sen2022}, enjoy essentially the same breakdown as univariate quantiles. Our work shows that the same result is true for a general class of strictly convex transportation costs. This is in sharp contrast to M-quantiles \citep{breckling:chambers1988} and  spatial quantiles \citep{chaudhuri1996,dudley:koltchinskii1992}, as the choice of the cost function in these cases  determines the breakdown point  of those multivariate quantiles \citep{konen:paindavein2025,konen:paindaveine2026}.

 \section{Main results}
 \subsection{Background in optimal transportation}
 
   The existence and uniqueness of Monge  maps for general cost functions was shown by \cite{GangboMcCann.96} under the following assumptions.  Let $\|\cdot\|$ denote the Euclidean norm and $c(x,y)=h(x-y)$ be such that  
\begin{assumptions}
\item\label{Stricl-convex} 
$h: \R^d\to [0, \infty)$ is strictly convex;
\item\label{COne} 
given a height $r>0$ and an angle $\theta \in (0,\pi) $, there exists some $M:=M(r, \theta)>0$ such that for all $\|{p} \|>M$, one can find a (truncated) cone 
	\begin{align*}
		\mathrm{Cone}(r, \theta, {z},{p}):=\left\lbrace {x}\in \R^d : \| {x}-{p}\|\| {z}\|\cos(\theta/2)\leq \left\langle {z},{x}-{p} \right\rangle\leq  r\| {z}\| \right\rbrace,
	\end{align*}
	with vertex at ${p}$ (and $ { z}\in \R^d \setminus \{{0}\}$) on which $h$ attains its maximum at ${p}$;
\item\label{Coercive} 
$\lim_{\|{x} \| \rightarrow \infty}\frac{h({x})}{\|{x} \| }= \infty $.
\end{assumptions}
Assumptions \ref{Stricl-convex} and \ref{Coercive} are easy to understand from an optimization standpoint. In contrast, Assumption \ref{COne} is geometric in nature as it requires that the level sets $ \{ z \in \mathbb{R}^d : h(z) \leq \lambda \}$ have a vanishing curvature as $\lambda\to\infty$. Note that the cost $c(x,y)=\|x-y\|^p$ satisfies this condition for $p>1$. Other relevant examples are the cost functions implicitly defined in \cite[p.~1151]{catoni2012}.

The characterization of Monge maps is typically achieved through the superdifferentials of $c$-concave functions; see \cite{Villani.09}. These can be seen as generalizations of subdifferentials of convex functions. Recall that a function $f \colon \mathbb{R}^d \to \mathbb{R} \cup \{-\infty\}$ is called \textit{$c$-concave} if it can be expressed as
\begin{equation}
    \label{CLTgeneq:c_concave}
f(x) = \inf_{(y, t) \in \mathcal{T}} \big\{ c(x, y) - t \big\},
\end{equation}
for some set $\mathcal{T} \subset \mathbb{R}^d \times \mathbb{R}$.  For such a function, the \textit{$c$-superdifferential} $\partial^c f$ is the set of pairs $(x, y) \in \mathbb{R}^d \times \mathbb{R}^d$ satisfying
\begin{equation}
    \label{eq:def-superdiff}
    f(z) \leq f(x) + \big[ c(z, y) - c(x, y) \big] \quad \text{for all } z \in \mathbb{R}^d.
\end{equation}
Note that in the case of the quadratic cost $h(x-y)=\|x-y\|^2$, one recovers the usual subdifferentials. Let $\mathcal{L}_d$ be the $d$-dimensional Lebesgue measure. 
We denote by $\partial^c f(x)$ the set of points $y$ such that $(x, y) \in \partial^c f$, and for any $U \subset \mathbb{R}^d$, we define $\partial^c f(U) = \bigcup_{x \in U} \partial^c f(x)$. For $\mathcal{L}_d$-a.e.~$x$ in the domain of $f$,  $\partial^c f(x)$ is a singleton (see \cite{GangboMcCann.96}). In this case we denote the unique element of $\partial^c f(x)$ by $\nabla^c f(x)$.  A set $\Gamma \subset \R^d \times \R^d$ is called $c$-\textit{cyclically monotone} if for any finite number of points $(x_1, y_1), (x_2, y_2), \dots, (x_N, y_N)\in \Gamma$, and for any permutation $\sigma$ of $\{1, \dots, N\}$, the following inequality holds:
\[
\sum_{i=1}^{N} c(x_i, y_i) \le \sum_{i=1}^{N} c(x_i, y_{\sigma(i)}).
\]
A celebrated result in convex analysis shows that the subdifferentials of convex functions are maximal cyclically monotone \citep{Rockafellar.70}. A similar relation was also shown to hold between  $c$-{cyclically monotone} sets and $c$-concave functions in \cite{Smith-Knott}; see also \cite[Theorem~2.7]{GangboMcCann.96}.
\begin{theorem}\label{theorem-c-cycl} A set $\Gamma\subset \R^d\times \R^d$ is $c$-cyclically 
    monotone if and only if there exists a $c$-concave function $f:\R^d\to \R\cup \{-\infty\} $ such that $\Gamma\subset \partial^c f$.  
\end{theorem}

Let $Q$ be absolutely continuous with respect to the $d$-dimensional Lebesgue measure, written as  $Q \ll \mathcal{L}_d$. The following  key result is due to \cite{GangboMcCann.96}. 

\begin{theorem}\label{Theorem-Gangbo-McCann}\cite[Theorems 3.7 and 4.5]{GangboMcCann.96}
    Assume that $Q\ll \mathcal{L}_d$ and that $c$ satisfies \ref{Stricl-convex}-\ref{Coercive}. Then,  there exists a unique $Q$-a.e.~defined Borel map $T_{Q\to P}$ such that 
\begin{enumerate}
    \item $T_{Q\to P}(x) \in \partial^c f_{Q\to P}(x)$  for $Q$-a.e.~$x\in \R^d$ and some $c$-concave $f_{Q\to P}$; 
    \item $T_{Q\to P}(x)=\nabla^c f_{Q\to P}(x) $ for $Q$-a.e.~$x\in \R^d$; 
    \item $ T_{Q\to P}$ pushes $Q$ forward to $P$. 
\end{enumerate} 
Moreover, if $\int_{\mathbb{R}^d} c({x},y)   dQ(x) d P (y)<\infty, $ then $T_{Q\to P}$ is the unique solution of \eqref{eq:Monge}.
\end{theorem}

\subsection{Breakdown point of the Monge map}
The breakdown point quantifies an estimator's robustness by measuring the minimum proportion of outliers required to make the estimator take arbitrarily aberrant values \citep{huber:ronchetti2009}.  
To define the notion of breakdown point of $ T_{Q\to P}(u)$ we need to have  $ T_{Q\to P}$  well-defined everywhere in ${\rm supp}(Q)$. However, the mapping  $ T_{Q\to P}$ is only defined $Q$-a.e., by \Cref{Theorem-Gangbo-McCann}. Then a natural extension of $ T_{Q\to P}(u)$ is given by any set-valued $c$-superdifferential   $$ {\bf T}_{Q\to P}=\partial^c f_{Q\to P}: \R^d\to \{A: A\subset \R^d\},$$
such that $\partial^c f_{Q\to P}(x)= \{ {T}_{Q\to P}(x)\}$ for $Q$-a.e.~$x\in \R^d$ %
;   see also Section 2.2 of  \cite{Avella-Gonzalez.25}. 
\begin{definition}
    The breakdown point of ${\bf T}_{Q\to P}(u)$  is defined as 
  $${\rm BP}({\bf T}_{Q\to P}(u), P)= \inf \left\{ \varepsilon\in (0,1): \sup_{{\bf T}_\varepsilon\in \Gamma_\varepsilon(P)} \sup_{v\in {\bf T}_\varepsilon (u)} \|v\| = \infty\right \},$$
where $\Gamma_\eps(P)$ denotes the set of ${\bf T}_{Q\to (1-\eps )P+\eps \mu}$ for some probability measure $\mu$. 
\end{definition}
 Now we can state the announced result. We recall that the Tukey depth \citep{Tukey.1975} of $u$ with respect to $Q$ is
$$ {\rm TD}(u;Q)=\inf_{v\in \mathbb{S}^{d-1}} Q(\{ x: \langle v, x-u\rangle \leq 0\}), $$
where $\mathbb{S}^{d-1}$ denotes the unit sphere. 
\begin{theorem}\label{theo:Breakdown}
Assume that $c$ satisfies \ref{Stricl-convex}-\ref{Coercive} and $Q \ll \mathcal{L}_d$. Fix   $u\in \R^d$  such that ${\rm TD}(u;Q)\in [0,1/2]$.     Then, 
   $$ {\rm BP}({\bf T}_{Q\to P}(u), P) = 
{\rm TD}(u;Q).$$ 
\end{theorem}

\begin{remark}[The direction of the perturbation]\label{rem:direction} For \(u \in \operatorname{supp}(Q)\), the Monge map \({\bf T}_{Q \to P}(u)\) breaks down with contamination in the direction of the mass point distribution
\(\delta_{z_n}\) (cf.~\Cref{lemma-Upper}), where 
\[
z_n = u - \nabla h^{*}(n\, v_u), 
\qquad 
v_u \in \arg\min_{v\in \mathbb{S}^{d-1}} Q\big(\{x : \langle v, x - u\rangle \le 0\}\big),
\]
and $h^*$ denotes the convex conjugate of $h$, i.e.,   $h^*(y) = \sup_{x \in \R^d} \{ \langle x, y \rangle - h(x) \}  $.  We recall here that $\nabla h^*$ is well defined (single-valued) as $h$ is strictly convex.
When \(h(x) = \|x\|^2\), the breakdown direction simplifies to \(u - n v_u\), as in 
\cite{Avella-Gonzalez.25} and \cite{paindaveine:passeggeri2024}. 
For radial costs, the contaminating direction is also of the form \(u - r_n v_u\) with \(r_n \to +\infty\). 
In general, however, the mapping \(\nabla h^{*}\) modifies this direction to reflect the geometry induced by the cost function \(h\).

\end{remark}
\subsection{Finite sample breakdown point}
We denote by $ Q_n$ (resp.~$P_n$) the empirical measure of a (deterministic or random) sequence of distinct points $u^{(n)}= \{u_1, \dots, u_n\}$ (resp.~$x^{(n)}=\{x_1, \dots, x_n\}$) such that $Q_n$ (resp.~$P_n$) converges to $Q$ (resp.~$P$)  in distribution. We are interested in the robustness properties of $ T_{Q_n\to P_n}(u_j)= x_{\sigma_{n}(j)}$, where\footnote{Note that $T_{Q_n\to P_n}$ solves the discrete version of \eqref{eq:Monge}.}  
\begin{equation}
    \label{eq:discrerte-OT}
    \sigma_n\in \argmin_{\sigma\in [[n]]} \sum_{i=1}^n c(u_i,x_{\sigma(i)}) 
\end{equation}
and $[[n]]$ denotes the set of permutations of $\{1, \dots, n\}$. From the definition of $T_{Q_n\to P_n}(u_j)$, it follows that the set $\{(u_j,T_{Q_n\to P_n}(u_j)\}_{j=1}^n=\{ (u_j, x_{\sigma_n(j)})\}_{j=1}^n $ is $c$-cyclically monotone, i.e., for every $S\subset \{1, \dots, n\}$ and every permutation  $\sigma$ of $S$, 
\begin{equation}
    \label{c-monotone-discrete}
   \sum_{i\in S} c(u_i,x_{\sigma_n(i)})  \leq \sum_{i\in S} c(u_i,x_{\sigma(i)}). 
\end{equation}
\begin{definition}
    The (replacement) finite sample breakdown point of $T_{Q_n\to P_n}(u_j)$  is defined as 
  $${\rm BP}({\bf T}_{Q_n\to P_n}(u), P)= \frac{1}{n}\cdot \min \left\{ \ell \in \{1 , \dots, n\}:  \sup_{\mu_n \in \mathcal{P}_\ell( P_n)} \|T_{Q_n\to \mu_n}(u_j)\| = \infty\right \},$$
  where $\mathcal{P}_\ell( P_n)$ denotes the set of empirical measures sharing at least $ n-\ell$ atoms with $P_n$.
\end{definition}
  We recall that a set of points  $u^{(n)}=\{u_1, \dots, u_n\}$ is in general position if no hyperplane contains  $d$ distinct points of 
  $u^{(n)}$.  We define the {\it lower Tukey depth} 
  $$ {\rm TD}^-(u_i;Q_n)=\inf_{v\in \mathbb{S}^{d-1}} Q_n(\{ x: \langle v, x-u_i\rangle < 0\}). 
  $$
   The following result characterizes the finite sample breakdown point of the OT problem.
   \begin{theorem}\label{Theorem:Discrete}
      Assume that $c$ satisfies \ref{Stricl-convex}-\ref{Coercive}.  Let $Q_n$ and $P_n$  be the empirical measures induced by $u^{(n)}$ and $x^{(n)}$ respectively.  Then, for every $j=1, \dots, n$, 
\begin{equation}
    \label{eq:discrete-BP}
     {\rm BP}({ T}_{Q_n\to P_n}(u_j), P_n) \in \left[ 
{\rm TD}^{-}(u_j;Q_n)+\frac{1}{n} , {\rm TD}(u_j;Q_n)\right] .
\end{equation}
If, furthermore, $u^{(n)}$ is in general position and $n\geq d$, then 
\begin{equation}
    \label{eq:General-Position-theorem}
    {\rm BP}({ T}_{Q_n\to P_n}(u_j), P_n) = 
{\rm TD}^{-}(u_j;Q_n)+\frac{1}{n}.
\end{equation}
  \end{theorem}
In \cite{Avella-Gonzalez.25} we showed \eqref{eq:discrete-BP} for the squared Euclidean distance cost. 
Hence, the relation \eqref{eq:General-Position-theorem} is sharper than the one we had in \cite{Avella-Gonzalez.25}. We relegate the proof of \Cref{Theorem:Discrete} to \Cref{app:proof_finite_sample_BP}. %

\section{Proof of the main results}
In this section we prove \cref{theo:Breakdown}. We first introduce some notation and recall some basic concepts from convex analysis.  Then we state two technical but fundamental lemmas. Finally we show the upper and lower bounds separately. 

We denote the open ball centered at $x$  with radius
$R$ is denoted  $ \mathbb{B}_{R}(x)$ and the unit sphere by $\mathbb{S}^{d-1}$. The topological support of a probability measure $P$ (the smallest closed set with $P$-probability one) is denoted by $ {\rm supp}(P) $. The interior, closure and boundary of a set $S$ are denoted by ${\rm int}(S)$, ${\rm cl}({S})$ and ${\rm bdry}(S)$, respectively. The convex hull of a set $S$ is denoted by ${\rm coh}(S)$.   We say that $S_1$ is compactly contained in $S_2$ if ${\rm cl}({S}_1)$ is compact and ${\rm cl}({S}_1)\subset {\rm int}(S_2)$. 
In this case, we write $S_1 \subset \subset S_2$. A sequence of sets $\{A_n\}_n$ {\it escapes to the horizon} if for every $R>0$, there exists $n_R$ such that $\mathbb{B}_R(0) \cap A_n=\emptyset $ for all $n\geq n_R$. In this case we write  $A_n\to \emptyset$.

The convex conjugate of a convex function $f:\R^d\to \R\cup \{+\infty\}$ is $f^*(y) = \sup_{x \in \R^d} \{ \langle x, y \rangle - f(x) \}  $, 
and its sub-differential is 
$$ \partial f(x) = \{ y \in \mathbb{R}^d: f(z) \geq f(x) + \langle y, z - x \rangle, \ \forall z\in \mathbb{R}^d \} . $$
For a function $f:\R^d\rightarrow \R\cup \{-\infty\}$ the \emph{$c$-conjugate of $f$}  is defined as
\begin{align}\label{CLTgeneq:c_conj}
f^c({y})=\inf_{{x}\in {\mathbb{R}^d}}\{ c({x},{y})-f({x})\} \ \ \text{for all  }{y}\in {\mathbb{R}^d}.
\end{align}
If $f:\R^d\rightarrow \R\cup \{-\infty\}$ is $c$-concave, then it follows that
$ f(x) + f^c(y)\leq c(x,y) $ with equality if and only if $y\in \partial^c f(x)$; cf.~\cite[Proposition~3.3.7]{RachevRueschendorf1998}. The domains of $ f$ and $\partial^c f$ are denoted as 
$$ {\rm dom}(f)= \{ x \in \R^d: f(x) \in \R\}\quad {\rm and}\quad {\rm dom}(\partial^c f)= \{ x \in \R^d: \partial^c f(x) \neq \emptyset \} . $$

\subsection{Two technical lemmas}
We state and prove the following two technical results, which play a key role in the derivation of \cref{theo:Breakdown}. 
The proofs are inspired by those of \cite[Lemmas~3.2 and~3.3]{delBarrioGonzalezLoubes.24}. For the geometric interpretation we refer to \Cref{fig:intuition-lemma}. 

\begin{figure}[h!]
    \centering
    
\tikzset{every picture/.style={line width=0.75pt}} %

\begin{tikzpicture}[x=0.75pt,y=0.75pt,yscale=-0.8,xscale=0.8]
\draw    (514.03,82.27) -- (496.03,145.87) ;
\draw [shift={(496.03,145.87)}, rotate = 285.8] [color={rgb, 255:red, 0; green, 0; blue, 0 }  ][line width=0.75]    (0,5.59) -- (0,-5.59)   ;
\draw [shift={(514.03,82.27)}, rotate = 285.8] [color={rgb, 255:red, 0; green, 0; blue, 0 }  ][line width=0.75]    (0,5.59) -- (0,-5.59)   ;
\draw  [fill={rgb, 255:red, 130; green, 6; blue, 185 }  ,fill opacity=0.12 ] (569.33,71.5) .. controls (589.33,61.5) and (679.33,51.5) .. (659.33,71.5) .. controls (639.33,91.5) and (639.33,101.5) .. (659.33,131.5) .. controls (679.33,161.5) and (589.33,161.5) .. (569.33,131.5) .. controls (549.33,101.5) and (549.33,81.5) .. (569.33,71.5) -- cycle ;
\draw  [fill={rgb, 255:red, 184; green, 233; blue, 134 }  ,fill opacity=0.12 ][dash pattern={on 0.84pt off 2.51pt}] (426,215.91) .. controls (390.11,244.86) and (306.37,209.77) .. (236.39,135.73) .. controls (170.46,65.95) and (139.75,-14.07) .. (162.31,-51.69) -- cycle ;
\draw  [fill={rgb, 255:red, 0; green, 0; blue, 0 }  ,fill opacity=0.02 ] (40,149.5) .. controls (40,98.41) and (103.35,57) .. (181.5,57) .. controls (259.65,57) and (323,98.41) .. (323,149.5) .. controls (323,200.59) and (259.65,242) .. (181.5,242) .. controls (103.35,242) and (40,200.59) .. (40,149.5) -- cycle ;
\draw    (104.67,25.17) -- (373,273) ;
\draw  [fill={rgb, 255:red, 0; green, 0; blue, 0 }  ,fill opacity=1 ] (234.25,146.4) .. controls (234.25,144.88) and (235.48,143.65) .. (237,143.65) .. controls (238.52,143.65) and (239.75,144.88) .. (239.75,146.4) .. controls (239.75,147.92) and (238.52,149.15) .. (237,149.15) .. controls (235.48,149.15) and (234.25,147.92) .. (234.25,146.4) -- cycle ;
\draw  [fill={rgb, 255:red, 189; green, 16; blue, 224 }  ,fill opacity=0.12 ] (203.3,87.93) .. controls (221.02,72.41) and (250.78,76.48) .. (273.18,98.69) .. controls (294.02,119.34) and (301.81,149.05) .. (293.56,170.56) -- cycle ;
\draw    (237,146.4) .. controls (252.76,105.03) and (449.18,83.48) .. (500.43,106.1) ;
\draw [shift={(502.67,107.17)}, rotate = 207.14] [fill={rgb, 255:red, 0; green, 0; blue, 0 }  ][line width=0.08]  [draw opacity=0] (8.93,-4.29) -- (0,0) -- (8.93,4.29) -- cycle    ;
\draw    (266.5,96.4) .. controls (264.69,46.92) and (522.29,43.91) .. (575.39,67.09) ;
\draw [shift={(577.67,68.17)}, rotate = 207.14] [fill={rgb, 255:red, 0; green, 0; blue, 0 }  ][line width=0.08]  [draw opacity=0] (8.93,-4.29) -- (0,0) -- (8.93,4.29) -- cycle    ;
\draw [color={rgb, 255:red, 65; green, 117; blue, 5 }  ,draw opacity=0.77 ][line width=3.75]    (355.67,219.17) -- (353.34,243.2) ;
\draw [shift={(352.67,250.17)}, rotate = 275.53] [fill={rgb, 255:red, 65; green, 117; blue, 5 }  ,fill opacity=0.77 ][line width=0.08]  [draw opacity=0] (20.54,-9.87) -- (0,0) -- (20.54,9.87) -- cycle    ;
\draw [color={rgb, 255:red, 65; green, 117; blue, 5 }  ,draw opacity=0.77 ][line width=3.75]    (156.67,9.17) -- (124.44,31.21) ;
\draw [shift={(118.67,35.17)}, rotate = 325.62] [fill={rgb, 255:red, 65; green, 117; blue, 5 }  ,fill opacity=0.77 ][line width=0.08]  [draw opacity=0] (20.54,-9.87) -- (0,0) -- (20.54,9.87) -- cycle    ;

\draw (458.4,144.8) node [anchor=north west][inner sep=0.75pt]  [font=\large] [align=left] {$\displaystyle \partial ^{c} f_{n}( u)$};
\draw (361.9,65.3) node [anchor=north west][inner sep=0.75pt]  [font=\large] [align=left] {$\displaystyle \partial ^{c} f_{n}$};
\draw (571.9,90.3) node [anchor=north west][inner sep=0.75pt]  [font=\large] [align=left] {$\displaystyle \partial ^{c} f_{n}( K)$};
\draw (219.2,151.13) node [anchor=north west][inner sep=0.75pt]  [font=\Large] [align=left] {$\displaystyle u$};
\draw (137,173.23) node [anchor=north west][inner sep=0.75pt]  [font=\LARGE] [align=left] {$\displaystyle Q $};
\draw (234,91.73) node [anchor=north west][inner sep=0.75pt]  [font=\Large] [align=left] {$\displaystyle K$};
\draw (129,73.23) node [anchor=north west][inner sep=0.75pt]  [font=\Large] [align=left] {$\displaystyle H$};
\draw (333,195) node [anchor=north west][inner sep=0.75pt]   [align=left] {$ $};
\draw (348.9,176.3) node [anchor=north west][inner sep=0.75pt]  [font=\Large] [align=left] {$\displaystyle R_{n_0}$};

\end{tikzpicture}
    \caption{ \small In this picture we provide a geometric description of \cref{lemma:technical}. If $\partial^c f_n (u)$ escapes to the horizon, then there exists a sequence  $\{(w_n,q_n)\}_n$ with $q_n\in \partial^c f_n(w_n)$ and such that $\|u-w_n\|\to 0$ and $\|q_n\|~\to~\infty$. Furthermore, part (ii) states that   $\partial^c f_n(K) $ (purple part on the right in the picture) escapes to the horizon,   for any compact set $K$ (purple part on the left in the picture)  contained in the green region $ R_{n_0}=\bigcap_{n\geq n_0}\{ x: h(x-q_n) \leq  h(w_n-q_n) \} $ for some $n_0$ large enough. Part (iii) states that the green region converges as $n_0\to \infty$ to a halfspace containing $u$. }
    \label{fig:intuition-lemma}
\end{figure}
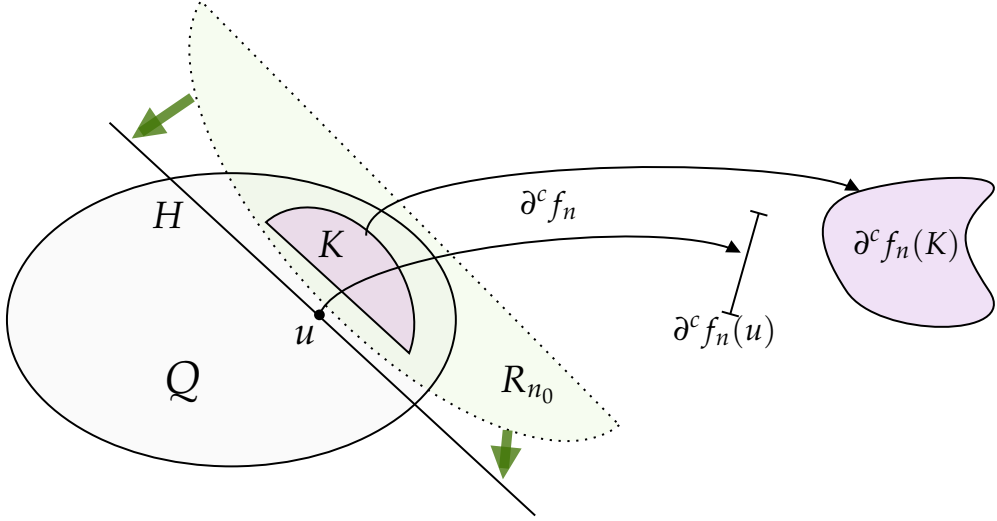

\begin{lemma}\label{lemma:technical}
    Let $f_n$ be a sequence of $c$-concave functions such that $\{\partial^c f_n\}_n \subset \R^{d}\times \R^d$ does not escape to the horizon along any subsequence. %
    Fix   $(u_n,p_n)\in\partial^c f_n  $ and assume that $\{u_n\}_n  \subset \subset  {\rm int}\left( \bigcap_n {\rm dom}(f_n) \right) $ is bounded
    and $\|p_n\|\to \infty$.
    Then the following statements hold:
    \begin{enumerate}
        \item  There exists a sequence  $\{(w_n,q_n)\}_n$ with $q_n\in \partial^c f_n(w_n)$ and such that $\|u_n-w_n\|\to 0$, $\|q_n\|~\to~\infty$ and $f_n(w_n)\to -\infty $.
        \item The sequence of sets $\{\partial^c f_{n}(K)\}_n$ escapes to the horizon, for any compact set $K$ such that for some $n_0\in \NN$, and sequences $\{(w_n,q_n)\}_n$ as defined in  $(i)$,
        \begin{equation}
        \label{eq:Compact-contained-intersect}
        K\subset  \bigcap_{n\geq n_0}\{ x \in \R^d: h(x-q_n) \leq  h(w_n-q_n) \} .
    \end{equation}
    \item   There exists $v\in \mathbb{S}^{d-1}$ such that  $\{\partial^c f_{n}(K)\}_n$ escapes to the horizon along a subsequence, for any compact set $K\subset \{z\in \R^d: \langle v, z -u\rangle<0\} $ where $u$ is any limit point of $\{u_n\}_n$. 
    \end{enumerate}

\end{lemma}
 We observe that \cref{lemma:technical} has a clear geometric interpretation illustrated in \Cref{fig:intuition-lemma}.
In the next subsection we will take $f_{n}= f_{Q\to (1-\varepsilon)P+\varepsilon \nu_n }$ for some  probability measure $\nu_n$ and we will use  \cref{lemma:technical} to find a lower bound on the breakdown point. Indeed, for this choice of  $f_{n}$, we have that $\{\partial^c f_n\}_n$ does not escape to the horizon along subsequences as required by \cref{lemma:technical}. This fact is consequence of our second technical lemma. 
\begin{lemma}\label{lemma:do-not-scape}
Fix $\varepsilon\in (0,1)$.     Let $f_{n}= f_{Q\to (1-\varepsilon)P+\varepsilon \nu_n }$ for some sequence $\{\nu_n\}_n$ of probability measures. Then $\{\partial^c f_n\}_n$ does not escape to the horizon along any subsequence.  
\end{lemma}

\subsection{Continuous breakdown point}
In this section we prove \cref{theo:Breakdown}. We divide the proof in two  lemmas. 
\begin{lemma}[Lower bound] \label{lemma:lower-bound} Fix $Q \ll \mathcal{L}_d$. Then, 
$$ {\rm BP}({\bf T}_{Q\to P}(u), P) \geq   
{\rm TD}(u, Q). $$
\end{lemma}
\begin{proof}
We only need consider ${\rm TD}(u;Q)>0$ since the result is trivial for ${\rm TD}(u;Q)=0$. Fix $u\in \R^d$ and $\varepsilon\in (0,1/2)$ with ${\rm TD}(u;Q)=\varepsilon+\alpha \in (0,1/2] $ for $\alpha>0$. Define $f_{n}= f_{Q\to (1-\varepsilon)P+\varepsilon \nu_n }$ for some sequence of probability measures $\{\nu_n\}_n$.  Note first that by \cite[Theorem~3.3 and Proposition~C.4]{GangboMcCann.96}, there exists a closed convex set $S$ such that
$$ 
{\rm int}(S)\subset {\rm dom}(f_n) \subset S \quad {\rm and}\quad {\rm int}(S)\subset {\rm dom}(\partial^c f_n) \subset S.
$$ We claim that $\partial^c f_n(u)\neq \emptyset $. To show the claim, we argue by contradiction. Assume that   $\partial^c f_n(u)= \emptyset $. Then the Hann-Banach theorem states that $\mathrm{int} (S)$ and $u$ are separated by a hyperplane $H$, which divides  $\R^d$ in two half-spaces: one closed $H_+$ containing $u$ and one open $H_-$ containing ${\rm int}(S)$. Since ${\rm TD}(u;Q)>0$, $Q$ gives some positive mass to both $H_+$ and $H_-$. However, since $\nabla^c f_n$  pushes $Q$ forward to a probability measure, there exists a dense subset of $S$  such that $Q(S)=1$. Since $Q \ll \mathcal{L}_d$, we also have $Q({\rm int}(S))=1$ and therefore   we contradict the fact that $Q(H_+)>0$. The claim follows.

Assume that  $\{f_{n}\}_n$
is such that $\|x_n\|\to \infty$ for some  $x_n\in \partial^c f_n(u) $. \cref{lemma:technical}~(iii) and \cref{lemma:do-not-scape} yield the existence of a subsequence $\{n_k\}_k$ and $v\in \mathbb{S}^{d-1}$ such that $\partial^c f_{n_k}(K) \to\emptyset$ for any  compact set $K\subset H_+= \{z\in\mathbb{R}^d: \langle v, z -u\rangle  \leq  0\} $. Since ${\rm TD}(u;Q)=\varepsilon+\alpha$,  it follows that $Q( H_+)\geq \varepsilon+\alpha$. 
Fix two compact sets $K_{\alpha,1}\subset H_+$ and $K_{\alpha,2}$ such that
$Q(K_{\alpha,1})\geq \varepsilon+\alpha/2  $
and 
$P(K_{\alpha,2})\geq 1-\alpha/3$.  As $\partial^c f_{n_k}(K_{\alpha,1})\cap K_{\alpha,2} =\emptyset $ for $k$ large enough,  we get the contradiction 
\begin{align*}
    \varepsilon+ \frac{\alpha}{2}&\leq Q(K_{\alpha,1}) \leq Q( (\partial^c f_n)^{-1} (\partial^c f_n(K_{\alpha,1})))=Q( (\nabla^c f_n)^{-1} (\partial^c f_n(K_{\alpha,1})))\\
    &=  ((1-\varepsilon)P+\varepsilon \nu_n)(\partial^c f_n(K_{\alpha,1}))\leq (1-\varepsilon)P(\mathbb{R}^d\setminus K_{\alpha,2})+\varepsilon \leq  \frac{\alpha}{3}  + \varepsilon .
\end{align*}
$$   $$
This proves that 
${\rm BP}({\bf T}_{Q\to P}(u), P) \geq   
{\rm TD}(u, Q). $
\end{proof}

The following result constructs  $f_{n} = f_{Q \to (1-\varepsilon)P + \varepsilon \nu_n}$ with a mass point distribution $\nu_n$ such that $\partial^c f_n(u)$ escapes to the horizon. This completes the proof of \cref{theo:Breakdown}. 
\begin{lemma}(Upper bound)\label{lemma-Upper}
Let $Q \ll \mathcal{L}_d$ be such that $Q({\rm bdry}({\rm supp}(Q)))=0$.
Fix $u\in \R^d$,  $\varepsilon\in (0,1)$ with ${\rm TD}(u;Q)=\varepsilon-\alpha$ for $\alpha>0$ and the contamination direction $\delta_{z_n}$, where 
$$z_n=u- \nabla h^* ( n\, v_u), \quad\quad v_u\in \argmin_{v\in \mathbb{S}^{d-1}} Q(\{ x\in\mathbb{R}^d: \langle v, x-u\rangle \leq 0\}).$$ 
    Then, letting    $f_{n}= f_{Q\to (1-\varepsilon)P+\varepsilon \delta_{z_n} }$, 
 $\{\partial^c f_n(u)\}_n$ escapes to the horizon. 
\end{lemma}
\begin{proof}  
The monotone convergence theorem yields
$$ \lim_{R\to \infty} Q\left( \left\{ x\in\mathbb{R}^d: \langle v_u, x-u\rangle > \frac{1}{R}\right\} \cap  {\rm cl}(\mathbb{B}_{R})\right) = 1-  \varepsilon+\alpha ,$$
which guarantees the existence of $R>0$ such that
\begin{equation}
    \label{eq:upper:bound-R}
    Q\left( \left\{ x\in\mathbb{R}^d: \langle v_u, x-u\rangle > \frac{1}{R}\right\} \cap  {\rm cl}(\mathbb{B}_{R})\right)\geq  1- \varepsilon+\frac{\alpha }{2}.  
\end{equation}
Since $\partial^c f_{n}(x)= \{\nabla^c f_n(x)\}$ for $Q$-a.e.~$x\in \R^d$ and  $(\nabla^c f_{n})\# Q= (1-\varepsilon)P+\varepsilon\, \delta_{z_n}$, the set
$$(\partial^c f_{n})^{-1}(z_n)=  \{ x:  z_n\in \partial^c f_{n}(x) \}
$$
has $Q$-measure at least $ \varepsilon$. As a consequence, \eqref{eq:upper:bound-R} implies  that 
$$S_n=  \left\{ x\in\mathbb{R}^d: \langle v_u, x-u\rangle > \frac{1}{R}\right\} \cap  {\rm cl}(\mathbb{B}_{R})\cap  (\partial^c f_{n})^{-1}(z_n) $$ 
is nonempty for all $n\in \NN$, so that  we can find a bounded sequence $\{x_n\}_n $ such that 
$ \langle v_u, x_n-u\rangle > 1/R $ and $z_n\in \partial^c f_n(x_n)$. Now we show that $\|y_n\|\to \infty$ for every sequence $\{y_n\}_n$ with $y_n\in \partial^c f_n(u)$. We argue by contradiction assuming that a subsequence $\{ y_{n_k}\}_{k} $, with $y_{n_k}\in \partial^c f_{n_k}(u)$, is bounded. As $h$ is continuous and $\{ x_{n_k}\}_{k} $  is also bounded, there exists $C>0$ such that, for every $k\in \NN$, 
$$ | h(u- y_{n_k})-h(x_{n_k}- y_{n_k}) |\leq C. $$
Hence, since $\partial^c f_{n_k}$ is $c$-monotone (cf.~\Cref{theorem-c-cycl}) and $(x_{n_k}, z_{n_k}), (u, y_{n_k}) \in \partial^c f_{n_k}$,   for every $k\in \NN$ %
\begin{equation}
    \label{eq:bound-continous-BP-breaking}
-C \leq  h(u- y_{n_k})-h(x_{n_k}- y_{n_k}) \leq h(u- z_{n_k}) -h(x_{n_k}- z_{n_k}).
\end{equation}
By definition of $z_{n_k}$ it follows that $n_k v_u \in \partial h(u- z_{n_k})$, which yields
\begin{align}\label{eq:contradiciton}
 h(u- z_{n_k}) \leq h(x_{n_k}- z_{n_k}) +n_k  \langle v_u,  u-x_{n_k} \rangle  \leq  h(x_{n_k}- z_{n_k}) - \frac{n_k}{R}, 
\end{align}
where we 
used  $ \langle v_u, x_{n_k}-u\rangle > 1/R $ for all $k\in \NN$. %
Taking limits in \eqref{eq:contradiciton} as $k\to \infty$, we contradict \eqref{eq:bound-continous-BP-breaking}. The result follows. 

\end{proof}

\section{Final remarks}

We believe that the equivalence between the breakdown of transport maps and the Tukey depth in the reference measure suggests a natural definition of the Tukey depth in metric spaces. Indeed, for a metric space $(\mathcal{M}, d)$,  we  could define the Tukey depth of a point $u \in \mathcal{M}$ with respect to a probability measure $Q$ on $\mathcal{M}$ as the breakdown point of the Monge map at $u$, when transporting $Q$ to itself (i.e., $P = Q$) with the squared distance cost $c(x,y) = d^2(x,y)$.

Our main result
\Cref{theo:Breakdown} establishes that the choice of cost function does not affect the robustness properties of the Monge map. This implies that the breakdown point is determined solely by intrinsic geometric properties of $\mathbb{R}^d$, rather than by the specific distance function employed. %
It would therefore be interesting to investigate whether different metrics $d$ can lead to the same definition of Tukey depth for general spaces and to study whether such a notion of depth  agrees with the metric Tukey depth proposed in \cite{LopezTukeyMetric}.

\appendix

\section{Finite sample breakdown point}
\label{app:proof_finite_sample_BP}

In this section we derive \Cref{Theorem:Discrete}. First we provide a useful remark, which allows us to apply \cref{lemma:technical}.  
\begin{remark}
\label{remark:discrete-extension}
Recall that $\mathcal{P}_\ell(P_n)$ denotes the set of empirical measures sharing at least $ n-\ell$ atoms with $P_n$. Fix $\{\mu_{k}\}_k\subset \mathcal{P}_\ell(P_n)$ and note that for each $k\in \NN$, since the set 
 $(u_j , { T}_{Q_n\to \mu_{k}}(u_j)\}_{j=1}^n $ is $c$-cyclically monotone, it is contained in the $c$-superdifferential $\partial^c f_{k}$ of a $c$-concave function $ f_{k}$ by \cref{theorem-c-cycl}. Furthermore, the set
$$\{(u_j,{ T}_{Q_n\to \mu_k}(u_j)) :   { T}_{Q_n\to \mu_k}(u_j) \in {\rm supp}(\mu_k) \cap {\rm supp}(P_n) \}$$
is nonempty (for $n-1\geq \ell\geq 1$) and bounded, so the sequence of sets $\{\partial^c f_{k}\}_k$ does not escape to the horizon along any subsequence. 

\end{remark}
The following result shows the lower bound. 
\begin{lemma}[Finite sample lower bound] \label{lem:finite_sample_lower}Under the conditions of \Cref{Theorem:Discrete}, we have  that  
$$ {\rm BP}({ T}_{Q_n\to P_n}(u_i), P_n) \geq   
{\rm TD}^-(u_i, Q_n)+\frac{1}{n}. $$
\end{lemma}
\begin{proof}
Fix $ i, \ell\in \{1, \dots, n\}$ and a sequence $  \{\nu_{k}\}_k\subset \mathcal{P}_\ell( P_n) $ such that  $\|{ T}_{Q_n\to \nu_k}(u_i)\|\to \infty$. We will show that $\ell\geq n\,{\rm TD}^-(u_i, Q_n)+1 $. 
Assume that ${\rm TD}^-(u_i, Q_n)>0$ as otherwise the result is trivial. 
 \Cref{remark:discrete-extension} shows that
$(u_j , { T}_{Q_n\to \nu_{k}}(u_j)\}_{j=1}^n $ is contained in the $c$-superdifferential $\partial^c f_{k}$ of a $c$-concave function $f_{k}$, for all $k\in \NN$. Since ${\rm TD}^-(u_i, Q_n)>0$, the separation theorem (see e.g.,~\cite[Theorem~11.2]{Rockafellar.70}) and $u^{(n)}\subset {\rm dom}(\partial^c f_k)$ yield 
$$ u_i \in {\rm int}({\rm coh}(u^{(n)})) \subset {\rm int}({\rm coh}({\rm dom}(\partial^c f_k) )), \quad \text{for all }k\in \NN.  $$
As in the proof of \Cref{lemma:lower-bound}, by \cite[Theorem~3.3 and Proposition~C.4]{GangboMcCann.96}, there exists a closed convex set $S$ such that
$$ 
{\rm int}(S)\subset {\rm dom}(f_k) \subset S \quad {\rm and}\quad {\rm int}(S)\subset {\rm dom}(\partial^c f_k) \subset S, \quad \text{for all }k\in \NN, 
$$ 
which, by \cite[Theorem~6.3]{Rockafellar.70}, implies that 
$$ u_i \in {\rm int}({\rm coh}(u^{(n)})) \subset {\rm int}(S)\subset {\rm dom}(\partial^c f_k)  . $$
This and \cref{remark:discrete-extension} allows us to apply  \cref{lemma:technical}. 
In particular, by \cref{lemma:technical}~(iii) there exists $v\in \mathbb{S}^{d-1}$ such that $\|{ T}_{Q_n\to \nu_k}(u_j)\| \to \infty$ along a subsequence for all  $u_j\in \{z\in\mathbb{R}^d: \langle  v, z -u_i\rangle<0 \}$. For simplicity, we use the slight abuse of notation $\nu_k$ for the subsequence. 
Assume that  $ {\rm supp}(\nu_k) \cap {\rm supp}(P_n)=\{x_1, \dots, x_{n-\ell}\}$ and call $\{z_{1,k}, \dots, z_{\ell,k}\}$ the remaining elements of  ${\rm supp}(\nu_k)$.  Since ${ T}_{Q_n\to \nu_k}$ is a bijection, it must follow that $\{z_{1,k}, \dots, z_{\ell,k}\}$ escapes to the horizon and that for $k$ large enough
$$ { T}_{Q_n\to \nu_k} (\{u_j: \langle v, u_j -u_i\rangle<0\} \cup \{u_i\}) \subset \{z_{1,k}, \dots, z_{\ell,k}\}. $$
We integrate with respect to $\nu_k$ to get 
\begin{align*}
{\rm TD}^-(u_i, Q_n)+\frac{1}{n}&=  \inf_{v'\in \mathbb{S}^{d-1}} Q_n(\{ x: \langle v', x-u_i\rangle < 0\})+\frac{1}{n}\\
&\leq Q_n(\{ x: \langle v, x-u_i\rangle < 0\})+\frac{1}{n}\\
&=\nu_k( { T}_{Q_n\to \nu_k} (\{u_j: \langle v, u_j -u_i\rangle<0\} \cup \{u_i\}) ) \leq \frac{\ell}{n} ,  
\end{align*}
which concludes the proof.
\end{proof}

Now we proceed with the upper bound. We start with the following fundamental lemma.

\begin{lemma}\label{lemma:discrete-upper}
    Fix $j\in \{1, \dots, n\}$, and $v\in \mathbb{S}^{d-1}$.  
Relabel  the reference sample $u^{(n)}=\{u_1, \dots, u_{n}\}= \{ u_{(1)}, \dots, u_{(n)}\}$ in such a way that for some $\ell\in \{1, \dots, n\}$,  $u_{(\ell)}= u_j$, 
$$ \{u_i\}_{i=1}^{\ell}= \{ x\in \mathbb{R}^d: \langle v, x-u_j\rangle \leq  0\} \cap u^{(n)}. $$
Fix $z_0\in \R^d$. For each $k\in \NN$ and $m\in \NN$, 
    define the measure 
    \begin{equation}\label{eq:def-nu-zeta}
        \nu_{k,v,m}=\frac{m}{n} \delta_{ z_{k,v,m}} + \frac{1}{n} \sum_{i=m+1}^n \delta_{T_{Q_n\to P_n}(u_{(i)})} , \quad \text{where}\quad z_{k,v,m} = u_{(\ell)} - \nabla h^{*}(z_0+k\, v). 
    \end{equation} Then the following hold: 

\begin{enumerate}
    \item  If   $ T_{Q_n\to \nu_{k,v,m}}(u_{(r)})=z_{k,v,m} $ for $k\in \NN$ large enough and some $r\geq \ell+1$, then
    $$ \text{$\| T_{Q_n\to \nu_{k,v,m}}(u_{j} )\|\to \infty  $ as $k\to \infty$.}$$
     
    \item If $m\geq \ell$, then $\| T_{Q_n\to \nu_{k,v,m}}(u_{j} )\|\to \infty  ,$ as $k\to \infty$. 
    \item ${\rm BP}({ T}_{Q_n\to P_n}(u_j), P_n) \leq  \frac{\ell}{n}$.
\end{enumerate}

\end{lemma}
\begin{proof}
   We use the simplified notation  $T_{k}=T_{Q_n\to \nu_{k,v,m}}$ and $T=T_{Q_n\to P_n}$. Assume that $ T_k(u_{(r)})= z_k:=z_{k,v,m} $ for $k$ large enough and some $r\geq \ell+1$. Then $\lambda= \langle v, u_{(r)}-u_{(\ell)}\rangle>0$. We show (i) {\it ad absurdum} and assume that $\{\|T_k(u_{(\ell)})\|\}_k$ is bounded.  Then there exists $C>0$ such that 
$ h( u_{(r)}-T_{k}(u_{(\ell)}))  \leq C  $ for all $k\in \NN$, so that by the $c$-cyclical monotonicity of $T_k$,  it follows that 
\begin{align*}
    -C&\leq    h( u_{(\ell)}-T_{k}(u_{(\ell)}))-h( u_{(r)}-T_{k}(u_{(\ell)}))\\
    &\leq  h( u_{(\ell)}-T_{k}(u_{(r)}))-h( u_{(r)}-T_{k}(u_{(r)})) =  h( u_{(\ell)}-z_k)-h( u_{(r)}-z_k) .
\end{align*}
By the convexity of $h$ and definition of $z_{k,v,m}$ (cf.~\eqref{eq:def-nu-zeta}), we have 
\begin{align*}
     h( u_{(\ell)}-z_k)-h( u_{(r)}-z_k)  &\leq  \langle z_0+  k\,v ,u_{(\ell)}-u_{(r)}\rangle = \langle z_0 ,u_{(\ell)}-u_{(r)}\rangle  -k \lambda \to -\infty,
\end{align*}
which gives a contradiction and  proves (i) because we are also assuming that $u_j=u_{(\ell)}$.  To obtain (ii), we observe that if $m\geq \ell $, then $T_{k}$ maps exactly $m$ points of $u^{(n)}$ to $z_k$. Hence, for $k$ large enough, we have $T_{k}(u_{(r)}) =z_k   $ for all $r\leq \ell$  or $T_{k}(u_{(s)}) =z_k   $   for some $ s>\ell$. If $T_{k}(u_{(s)}) =z_k   $   for some $ s>\ell$,  part~(i) implies that   $\|T_{k}(u_{(\ell)})\|\to \infty   $. Hence, to  conclude (ii), we need to show that $\|z_k\|\to \infty$. Assume that   $\{z_k\}_k$ admits a bounded  subsequence. Then, by \cite[Proposition~16.20]{Bauschke.2017.convex},  
 $\{\nabla h^{*}(z_0+k\, v)\}_k$ admits a bounded subsequence, hence 
  $z_0+kv \in \partial h (\nabla h^{*}(z_0+k\, v))$ admits  a bounded subsequence, which is a contradiction. {\it A fortiori}, $\|z_k\|\to \infty$ and (ii) follows. 

Part (iii) follows by noticing that $ \nu_{k,v,\ell}$ shares $n-\ell$ atoms with $ P_n$. 
\end{proof}
Now we can prove the upper bounds in \Cref{Theorem:Discrete}
as a consequence of \Cref{lemma:discrete-upper}. 
\begin{lemma}[Finite sample upper bound]\label{lem:finite_sample_UB}
For every $j=1, \dots, n$, 
\begin{equation}
    \label{eq:discrete-BP-proof}
     {\rm BP}({ T}_{Q_n\to P_n}(u_j), P_n) \leq  {\rm TD}(u_j;Q_n) .
\end{equation}
If, furthermore, $u^{(n)}=\{u_1, \dots, u_n\}$ is in general position, then 
\begin{equation}
    \label{eq:General-Position-proof}
    {\rm BP}({ T}_{Q_n\to P_n}(u_j), P_n) \leq  
{\rm TD}^{-}(u_j;Q_n)+\frac{1}{n}.
\end{equation}
\end{lemma}

\begin{proof}
In order to prove \eqref{eq:discrete-BP-proof}, we apply \Cref{lemma:discrete-upper}~(iii) to the vector 
$$v_*\in \argmin_{v\in \mathbb{S}^{d-1}} Q_n (\{ x: \langle v, x-u_j\rangle \leq  0\}).$$
and  $\ell=\sum_{i=1}^n {\bf 1}_{\langle v_*, u_i-u_j\rangle \leq  0}$. Since $\ell/n={\rm TD}(u_j; Q_n)$, we have found the upper bound \eqref{eq:discrete-BP-proof}. 

 To prove \eqref{eq:General-Position-proof}, define 
 $$w\in \argmin_{v\in \mathbb{S}^{d-1}} Q_n (\{ x: \langle v, x-u_j\rangle <  0\}).$$
 We will apply \cref{lemma:discrete-upper} to a  continuous perturbation of  $w$ defining a hyperplane containing exactly one more point than $\{ x: \langle w, x-u_j\rangle <  0\}$. The intuition for this is simple: the general position assumption implies that no more than 2 points can lie in the same line. 
 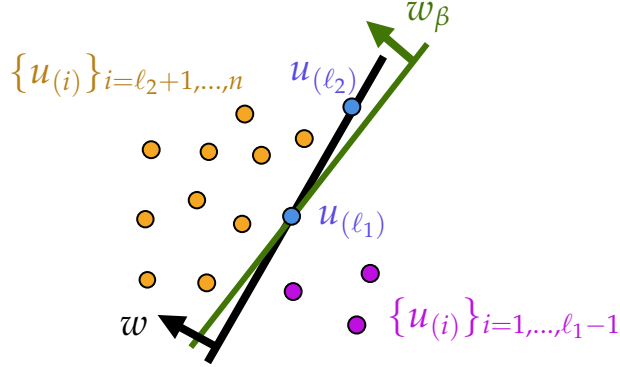
\begin{figure}[h!]
    \centering

\tikzset{every picture/.style={line width=0.75pt}} %

\begin{tikzpicture}[x=0.6pt,y=0.6pt,yscale=-1,xscale=1]
\draw [color={rgb, 255:red, 0; green, 0; blue, 0 }  ,draw opacity=1 ][line width=3]    (375.83,42.08) -- (264.83,234.92) ;
\draw  [fill={rgb, 255:red, 245; green, 166; blue, 35 }  ,fill opacity=1 ] (259,184.83) .. controls (259,181.89) and (261.39,179.5) .. (264.33,179.5) .. controls (267.28,179.5) and (269.67,181.89) .. (269.67,184.83) .. controls (269.67,187.78) and (267.28,190.17) .. (264.33,190.17) .. controls (261.39,190.17) and (259,187.78) .. (259,184.83) -- cycle ;
\draw  [fill={rgb, 255:red, 74; green, 144; blue, 226 }  ,fill opacity=1 ] (350.5,73.83) .. controls (350.5,70.89) and (352.89,68.5) .. (355.83,68.5) .. controls (358.78,68.5) and (361.17,70.89) .. (361.17,73.83) .. controls (361.17,76.78) and (358.78,79.17) .. (355.83,79.17) .. controls (352.89,79.17) and (350.5,76.78) .. (350.5,73.83) -- cycle ;
\draw  [fill={rgb, 255:red, 189; green, 16; blue, 224 }  ,fill opacity=1 ] (313.5,190.33) .. controls (313.5,187.39) and (315.89,185) .. (318.83,185) .. controls (321.78,185) and (324.17,187.39) .. (324.17,190.33) .. controls (324.17,193.28) and (321.78,195.67) .. (318.83,195.67) .. controls (315.89,195.67) and (313.5,193.28) .. (313.5,190.33) -- cycle ;
\draw  [fill={rgb, 255:red, 245; green, 166; blue, 35 }  ,fill opacity=1 ] (293.5,104.33) .. controls (293.5,101.39) and (295.89,99) .. (298.83,99) .. controls (301.78,99) and (304.17,101.39) .. (304.17,104.33) .. controls (304.17,107.28) and (301.78,109.67) .. (298.83,109.67) .. controls (295.89,109.67) and (293.5,107.28) .. (293.5,104.33) -- cycle ;
\draw [color={rgb, 255:red, 65; green, 117; blue, 5 }  ,draw opacity=1 ][line width=2.25]    (403.33,34.08) -- (254.33,225.42) ;
\draw  [fill={rgb, 255:red, 74; green, 144; blue, 226 }  ,fill opacity=1 ] (312.5,142.83) .. controls (312.5,139.89) and (314.89,137.5) .. (317.83,137.5) .. controls (320.78,137.5) and (323.17,139.89) .. (323.17,142.83) .. controls (323.17,145.78) and (320.78,148.17) .. (317.83,148.17) .. controls (314.89,148.17) and (312.5,145.78) .. (312.5,142.83) -- cycle ;
\draw [line width=3]    (238.61,207.95) -- (270,225) ;
\draw [shift={(233.33,205.08)}, rotate = 28.51] [fill={rgb, 255:red, 0; green, 0; blue, 0 }  ][line width=0.08]  [draw opacity=0] (16.97,-8.15) -- (0,0) -- (16.97,8.15) -- cycle    ;
\draw [color={rgb, 255:red, 65; green, 117; blue, 5 }  ,draw opacity=1 ][line width=3]    (369.38,23.49) -- (394.5,45.08) ;
\draw [shift={(364.83,19.58)}, rotate = 40.68] [fill={rgb, 255:red, 65; green, 117; blue, 5 }  ,fill opacity=1 ][line width=0.08]  [draw opacity=0] (16.97,-8.15) -- (0,0) -- (16.97,8.15) -- cycle    ;
\draw  [fill={rgb, 255:red, 245; green, 166; blue, 35 }  ,fill opacity=1 ] (224,100.83) .. controls (224,97.89) and (226.39,95.5) .. (229.33,95.5) .. controls (232.28,95.5) and (234.67,97.89) .. (234.67,100.83) .. controls (234.67,103.78) and (232.28,106.17) .. (229.33,106.17) .. controls (226.39,106.17) and (224,103.78) .. (224,100.83) -- cycle ;
\draw  [fill={rgb, 255:red, 245; green, 166; blue, 35 }  ,fill opacity=1 ] (283,78.33) .. controls (283,75.39) and (285.39,73) .. (288.33,73) .. controls (291.28,73) and (293.67,75.39) .. (293.67,78.33) .. controls (293.67,81.28) and (291.28,83.67) .. (288.33,83.67) .. controls (285.39,83.67) and (283,81.28) .. (283,78.33) -- cycle ;
\draw  [fill={rgb, 255:red, 245; green, 166; blue, 35 }  ,fill opacity=1 ] (253,132.67) .. controls (253,129.81) and (255.31,127.5) .. (258.17,127.5) .. controls (261.02,127.5) and (263.33,129.81) .. (263.33,132.67) .. controls (263.33,135.52) and (261.02,137.83) .. (258.17,137.83) .. controls (255.31,137.83) and (253,135.52) .. (253,132.67) -- cycle ;
\draw  [fill={rgb, 255:red, 245; green, 166; blue, 35 }  ,fill opacity=1 ] (281.5,147.67) .. controls (281.5,144.81) and (283.81,142.5) .. (286.67,142.5) .. controls (289.52,142.5) and (291.83,144.81) .. (291.83,147.67) .. controls (291.83,150.52) and (289.52,152.83) .. (286.67,152.83) .. controls (283.81,152.83) and (281.5,150.52) .. (281.5,147.67) -- cycle ;
\draw  [fill={rgb, 255:red, 245; green, 166; blue, 35 }  ,fill opacity=1 ] (220.5,144.67) .. controls (220.5,141.81) and (222.81,139.5) .. (225.67,139.5) .. controls (228.52,139.5) and (230.83,141.81) .. (230.83,144.67) .. controls (230.83,147.52) and (228.52,149.83) .. (225.67,149.83) .. controls (222.81,149.83) and (220.5,147.52) .. (220.5,144.67) -- cycle ;
\draw  [fill={rgb, 255:red, 245; green, 166; blue, 35 }  ,fill opacity=1 ] (260.5,102.17) .. controls (260.5,99.31) and (262.81,97) .. (265.67,97) .. controls (268.52,97) and (270.83,99.31) .. (270.83,102.17) .. controls (270.83,105.02) and (268.52,107.33) .. (265.67,107.33) .. controls (262.81,107.33) and (260.5,105.02) .. (260.5,102.17) -- cycle ;
\draw  [fill={rgb, 255:red, 245; green, 166; blue, 35 }  ,fill opacity=1 ] (320.5,93.83) .. controls (320.5,90.89) and (322.89,88.5) .. (325.83,88.5) .. controls (328.78,88.5) and (331.17,90.89) .. (331.17,93.83) .. controls (331.17,96.78) and (328.78,99.17) .. (325.83,99.17) .. controls (322.89,99.17) and (320.5,96.78) .. (320.5,93.83) -- cycle ;
\draw  [fill={rgb, 255:red, 189; green, 16; blue, 224 }  ,fill opacity=1 ] (362,178.92) .. controls (362,175.93) and (364.43,173.5) .. (367.42,173.5) .. controls (370.41,173.5) and (372.83,175.93) .. (372.83,178.92) .. controls (372.83,181.91) and (370.41,184.33) .. (367.42,184.33) .. controls (364.43,184.33) and (362,181.91) .. (362,178.92) -- cycle ;
\draw  [fill={rgb, 255:red, 189; green, 16; blue, 224 }  ,fill opacity=1 ] (353.5,211.42) .. controls (353.5,208.43) and (355.93,206) .. (358.92,206) .. controls (361.91,206) and (364.33,208.43) .. (364.33,211.42) .. controls (364.33,214.41) and (361.91,216.83) .. (358.92,216.83) .. controls (355.93,216.83) and (353.5,214.41) .. (353.5,211.42) -- cycle ;
\draw  [fill={rgb, 255:red, 245; green, 166; blue, 35 }  ,fill opacity=1 ] (222,182.92) .. controls (222,180.2) and (224.2,178) .. (226.92,178) .. controls (229.63,178) and (231.83,180.2) .. (231.83,182.92) .. controls (231.83,185.63) and (229.63,187.83) .. (226.92,187.83) .. controls (224.2,187.83) and (222,185.63) .. (222,182.92) -- cycle ;

\draw (314.5,43.33) node [anchor=north west][inner sep=0.75pt]  [font=\Large,color={rgb, 255:red, 89; green, 74; blue, 226 }  ,opacity=1 ] [align=left] {$\displaystyle u_{( \ell _{2})}$};
\draw (388,5.33) node [anchor=north west][inner sep=0.75pt]  [font=\Large,color={rgb, 255:red, 65; green, 117; blue, 5 }  ,opacity=1 ] [align=left] {$\displaystyle w_{\beta }$};
\draw (207.5,204.5) node [anchor=north west][inner sep=0.75pt]  [font=\Large,color={rgb, 255:red, 0; green, 0; blue, 0 }  ,opacity=1 ] [align=left] {$\displaystyle w$};
\draw (332.5,134.33) node [anchor=north west][inner sep=0.75pt]  [font=\Large,color={rgb, 255:red, 89; green, 74; blue, 226 }  ,opacity=1 ] [align=left] {$\displaystyle u_{( \ell _{1})}$};
\draw (376,191.83) node [anchor=north west][inner sep=0.75pt]  [font=\Large,color={rgb, 255:red, 189; green, 16; blue, 224 }  ,opacity=1 ] [align=left] {$\displaystyle \{u_{( i)}\}_{i=1,\dotsc ,\ell _{1} -1}$};
\draw (137.5,37.83) node [anchor=north west][inner sep=0.75pt]  [font=\Large,color={rgb, 255:red, 184; green, 126; blue, 22 }  ,opacity=1 ] [align=left] {$\displaystyle \{u_{( i)}\}_{i=\ell _{2} +1,\dotsc ,n}$};

\end{tikzpicture}
    \caption{ \small Illustration of the geometric argument used in the proof of \Cref{lem:finite_sample_UB}. If $w$   defines a half-space with  border passing through more than one point, then we can find a slight modification  $w_\beta$  defining a half-space with the same number of interior points, but with  border passing only through $u_{(\ell_1)}$.  }
    \label{fig:intuition-lemma-2}
\end{figure}

 Relabel the reference sample $u^{(n)}$ in such a way that for some ${\ell}_1,\ell_2\in \{1, \dots, n\}$,  $u_{(\ell_1)}= u_j$,
\begin{equation}
    \label{eq:setting-general-position}
     \{u_{(i)}\}_{i=1}^{\ell_1-1}= \{ x: \langle w, x-u_{(\ell_1)}\rangle < 0\} \cap u^{(n)} \ \text{and}\   \{u_{(i)}\}_{i=\ell_1}^{\ell_2}= \{ x: \langle w, x-u_{(\ell_1)}\rangle =  0\} \cap u^{(n)}.
\end{equation}
Note that by construction
\begin{equation}
    \label{eq:ell-1}
   \frac{ \ell_1}{n}=  {\rm TD}^{-}(u_{(\ell_1)};Q_n)+\frac{1}{n}=  {\rm TD}^{-}(u_j;Q_n)+\frac{1}{n}.
\end{equation}
 As $u^{(n)}$ is in general position, it follows that $ u_{(\ell_1)}\notin {\rm aff}( \{u_{(i)}\}_{i=\ell_1+1}^{\ell_2})$, where for a set $S$,   ${\rm aff}(S)$  denotes the smallest affine subspace containing $S$. Since  $u^{(n)}$ is in general position,  it follows that ${\rm aff}( \{u_{(i)}\}_{i=\ell_1+1}^{\ell_2})$ has  dimension  at most $d-2$ and  
 $$ {\rm aff}( \{u_{(i)}\}_{i=\ell_1+1}^{\ell_2})\subset  \{ x: \langle w, x-u_{(\ell_1)}\rangle =  0\}. $$ 
 Then there exists $a>0$,  $v_0 \in \mathbb{S}^{d-1} $ orthogonal to $ w$ and such that 
$$ {\rm aff}(\{u_{(i)}\}_{i=\ell_1+1}^{\ell_2})\subset \{  x: \langle  v_0,x-u_{(\ell_1)}  \rangle =a \} \quad {\rm and}\quad u_{(\ell_1)} \in \{  x: \langle v_0, x -u_{(\ell_1)} \rangle =0 \} . $$
Define the continuous function 
$$ \R^d\times \R \ni (x,\beta)\mapsto f(x,\beta):= \langle    w_\beta, x -u_{(\ell_1)} \rangle, \quad {\rm for }\quad  w_\beta =\frac{w+ \beta v_0}{\|w+ \beta v_0\|}.   $$
Note that, for every $\beta>0$, 
$$ \{u_{(i)}\}_{i=\ell_1+1}^{\ell_2} \subset  \{  x: f(x,\beta) >0  \} \quad {\rm and}\quad  u_{(\ell_1)} \in \{  x: f(x,\beta) =0  \}, $$
while for $\beta=0$ we recover \eqref{eq:setting-general-position}. Therefore,
$$ \{u_{(i)}\}_{i=1}^{\ell_1-1} \subset  \{  x: f(x,0) <0  \},  \quad \{u_{(i)}\}_{i=\ell_2+1}^{n} \subset  \{  x: f(x,0)>0  \} $$ 
and $ \{u_{(j)}\}_{j=\ell_1}^{\ell_2}$ are the unique elements of the reference sample $u^{(n)}$ in $\{  x: f(x,0) =0  \}$. Since $f$ is continuous, for every $i\leq \ell_1-1$ (resp.~$i\geq \ell_2+1$), there exists $\beta_{(i)}>0$ such that   $u_{(i)} \in \{  x: f(x,\beta) <0  \} $ (resp.~$u_{(i)} \in \{  x: f(x,\beta) >0  \} $) for all $|\beta|\leq \beta_{(i)}$. Taking $\beta_*=\min_{i} \beta_{(i)}$, we conclude that 
$$ \{u_{(i)}\}_{i=1}^{\ell_1-1} \subset  \{  x: f(x,\beta_*) <0  \},  \quad  \{u_{(i)}\}_{i=\ell_1+1}^{n} \subset  \{  x: f(x,\beta_*) >0  \}  $$
 and $ u_{(\ell_1)}$ is the unique element of the reference sample $u^{(n)}$ in $\{  x: f(x,\beta_*) =0  \}$. Hence we can apply \cref{lemma:discrete-upper} with $v=w_{\beta_*}$ and $m= \ell_1$ and conclude that 
$$  {\rm BP}({ T}_{Q_n\to P_n}(u_j), P_n) \leq  
\frac{\ell_1}{n}={\rm TD}^{-}(u_j;Q_n)+\frac{1}{n}, $$
where the last equality is \eqref{eq:ell-1}. 

\end{proof}
\section{Proof of \Cref{lemma:technical} and \Cref{lemma:do-not-scape} }

\begin{proof}[Proof of \Cref{lemma:technical}] The proof is divided in several steps. 

{\it Step 1: $f_n$ and $f_n^c$ are locally upper bounded.   } As $\{\partial^c f_n\}_n$ does not
escape to the horizon along any subsequence, there exists a bounded sequence  $\{(x_n,y_n)\}_n$ with $(x_n,y_n)\in \partial^c f_n$.  We can assume without losing generality that $f_n^c(y_n)=0$ for all $n\in \NN$. (Note that $\partial^c f_n = \partial^c (f_n+a) $ and $ f_n^c-a =  (f_n+a)^c $  for all $a\in \R$. ) Then the relation (cf.~\cite[Proposition~3.3.7]{RachevRueschendorf1998}) 
$$ f_n(x_n)= f_n(x_n) + f_n^c(y_n) = h(x_n-y_n) $$
implies that $|f_n(x_n)|$ is bounded. Furthermore, for all $x\in \R^d$ there exists a constant $C(\|x\|)$, depending exclusively on $\|x\|$,  such that 
$$ f_n(x)= f_n(x)+ f_n^c(y_n) \leq   h(x-y_n)  \leq C(\|x\|) , $$
where the first inequality follows from the definition of $f_n^c$. 
The same argument shows that, for every $y\in \R^d$,  
\begin{equation}
    \label{eq:proof-lemma-tech-upper-bound-conjugate}
     f_n^c(y)\leq  |f_n(x_n) |+   h(x_n-y)  \leq C(\|y\|) .
\end{equation}

{\it Step 2: Proof of (i).} 
Fix $\epsilon>0$ and define $v_n= u_n-p_n $ and  $z_n= u_n- \epsilon \frac{v_n}{\|v_n\|} $ for $\epsilon>0$ small so that $ z_n \in  {\rm int}\left( \bigcap_n {\rm dom}(f_n) \right) $ for all $n$. Since $(u_n,p_n)\in\partial^c f_n  $, we have (see \eqref{eq:def-superdiff})
$$ f_n({z}_n)\leq f_n(u_n)+[h({z}_n-p_n)-h(u_n-p_n)], $$
which implies
\begin{align*}
  C-f_n(z_n)&\geq   f_n(u_n)-f_n({z}_n)\\
  &\geq h(u_n-p_n) -h({z}_n-p_n) = h(v_n) -h\left(\left(1-\frac{\epsilon}{\|v_n\|} \right) v_n \right) ,
\end{align*}
where $C=C(\sup_n\|u_n\|)$.  Since $h$ is convex, for any  $s_n\in \partial h\left(\left(1-\frac{\epsilon}{\|v_n\|} \right) v_n \right)$, 
$$ f_n(z_n) \leq C+ h\left(\left(1-\frac{\epsilon}{\|v_n\|} \right) v_n \right) -h(v_n) \leq C -\epsilon \left\langle s_n, \frac{v_n}{\|v_n\|} \right\rangle .$$
Again by convexity of $h$, we have 
$$ {\left\langle s_n, \left(1-\frac{\epsilon}{\|v_n\|} \right) v_n \right\rangle}   \geq  {h\left(\left(1-\frac{\epsilon}{\|v_n\|} \right) v_n \right) -h(0)},$$
so that letting $n$ big enough in order that $\frac{\epsilon}{\|v_n\|}<\frac{1}{2}$,  we get
$$ {\left\langle s_n, \frac{v_n}{\|v_n\|}\right\rangle}  \geq  \frac{h\left(\left(1-\frac{\epsilon}{\|v_n\|} \right) v_n \right) -h(0)}{\left(1-\frac{\epsilon}{\|v_n\|} \right) \|v_n\| }  \to +\infty ,  $$
where the limit follows from \ref{Coercive}. Hence,  for every $\epsilon>0$ there exists a sequence $\{w_{n,\epsilon}\}_n$ (taking $w_{n,\epsilon}= u_n- \epsilon \frac{v_n}{\|v_n\|} $)  such that 
$$ w_{n,\epsilon}\in {\rm int}\!\left(\bigcap_n {\rm dom}(f_n)\right),\quad   
\|u_n-w_{n,\epsilon}\|\leq \epsilon  \quad \text{and}\quad  f_n(w_{n,\epsilon})\to -\infty. $$
In particular, for each $m\in\mathbb{N}$ we can construct points 
$\tilde{w}_{n,m}$ and an index $n_m$ such that 
$\|u_n-\tilde{w}_{n,m}\|\leq 1/m$ and $f_n(\tilde{w}_{n,m})\leq -m$ whenever 
$n\geq n_m$. Setting $w_n'=\tilde{w}_{n,n}$, we obtain a sequence with 
\begin{equation}
    \label{eq:w.prima}
    \|u_n-w_n' \|\leq 1/n\quad {\rm and}\quad f_n(w_n')\leq -n,\quad \text{for all } n\in \NN. 
\end{equation}
 Since $w_n'$ eventually 
lies in the interior of ${\rm dom}(f_n)$, Theorem~3.3 and Proposition~C.4 
in \cite{GangboMcCann.96} guarantee that 
$\partial^c f_n(w_n')\neq \emptyset$.

We now claim that either $f_n(u_n)\to -\infty$ along a subsequence or that  
$\{\partial^c f_n(w_n')\}_n$ escapes to the horizon. Indeed, if 
$\{\partial^c f_n(w_n')\}_n$ does not escape to the horizon, then there exists a bounded 
subsequence $\{q_{n_k}\}_k$ with $q_{n_k}\in \partial^c f_n(w_{n_k}')$. In that case, 
the inequality
\begin{equation}
    \label{eq:diverges-one-or-other}
    f_{n_k}(u_{n_k})\leq f_{n_k}(w_{n_k}')+\bigl(h(u_{n_k}-q_{n_k})-h(w_{n_k}'-q_{n_k})\bigr) \leq -n_k+\bigl(h(u_{n_k}-q_{n_k})-h(w_{n_k}'-q_{n_k})\bigr)
\end{equation}
implies that $f_n(u_{n_k})\to -\infty$, which proves the claim. Hence, the sequence 
$$ w_n = \begin{cases}
   u_n  & \text{if }     -f_n(u_n) \geq \min_{y\in \partial^c f_n(w_{n}')} \|y\| ,\\
   w_n'   & \text{otherwise},    
\end{cases}   $$
satisfies the conclusion of (i). 
To check this claim, we argue by contradiction. First, if $f_{n_k}(w_{n_k})\to -C$ for some subsequence, then by \eqref{eq:w.prima}, $w_{n_k}=u_{n_k}$. Hence, $-f_{n_k}(u_{n_k}) \geq \min_{y\in \partial^c f_{n_k}(w_{n_k}')} \|y\|$, which contradicts \eqref{eq:diverges-one-or-other}. Second, if now $\partial f_{n_k}(w_{n_k})$ is bounded for some subsequence, as $ \|p_n\|\to \infty$ and $(u_n,p_n)\in\partial^c f_n  $, it must follow that $w_{n_k}=w'_{n_k}$ -- or equivalently $\min_{y\in \partial^c f_{n_k}(w_{n_k}')} \|y\| > -f_{n_k}(u_{n_k}) $. Hence,   $\min_{y\in \partial^c f_{n_k}(w_{n_k}')} \|y\| $ is bounded and {\it a fortiori} $f_{n_k}(u_{n_k})$ cannot diverge. This again contradicts \eqref{eq:diverges-one-or-other}. 

{\it Step 3: Proof of (ii).}  Let $K$ be as in \eqref{eq:Compact-contained-intersect}. Since $(w_n,q_n)\in \partial^c f_n$, for every $x\in K$,
   $$f_n(x)\leq f_n(w_n)+[h(x-q_n)-h(w_n-q_n)] \leq f_n(w_n), $$  
so that
\begin{equation}
    \label{eq:sup-fninK}
    \sup_{x\in K} f_n(x) \to -\infty. 
\end{equation}
 For any $y_x\in \partial^c f(x),$ it follows from \eqref{eq:proof-lemma-tech-upper-bound-conjugate} that 
$$ f_n(x)+ C( \|y_x\|)\geq   f_n(x)+ f_n^c(y_x) = h(x-y_x) \geq 0,$$
which implies 
$ f_n(x) \geq - C( \|y_x\|) $.  From here and \eqref{eq:sup-fninK}, we deduce that $\{\partial^c f_{n}(K)\}_n$ escapes to the horizon. 

{\it Proof of (iii).} Note that the limit points of $\{u_n\}_n$ are those of $\{w_n\}_n$. Let $u$ be a limit point of $\{w_n\}_n$. 
For each $n\in \NN$ let $r_n$ be the largest $r>0$ such that there exists a direction $v_n\in \mathbb{S}^{d-1}$ satisfying 
$$ {\rm Cone}\left(r_n, \frac{ \pi }{1+r_n^{-1}}, {v}_n,w_n-q_n\right)  \subset \{ x: h(x) \leq  h(w_n-q_n) \} .$$
 After taking  subsequences, we can assume that $w_n\to u$ and $v_n\to v_*=-v$.  Fix a compact set  $$K\subset \{z: \langle v_*, z -u\rangle>0\}=  \{z: \langle v, z -u\rangle<0\}.$$ By \ref{COne} and the fact that $\|q_n\|\to \infty$,  it follows that $r_n\to +\infty$ as $n\to \infty$, which implies 
\begin{equation}
    \label{eq:cos-to-zero}
    \cos\left( \frac{ \pi/2 }{1+r_n^{-1}} \right) \to 0.
\end{equation}
 Since $K$ is compact, there exists some $\alpha>0$ such that $K\subset \{z: \langle v_*, z -u\rangle\geq \alpha\} $. 
Hence,  for every $x\in K$, 
\begin{align*}
    \left\langle v_n, x-w_n \right\rangle &= \left\langle v_n-v_*, x-w_n \right\rangle+ \left\langle v_*, {x}-w_n \right\rangle\\
    &\geq \left\langle v_*, x-u \right\rangle-\|v_n-v_*\|  \| x-w_n\| - \| w_n-u\|\geq  \alpha -\|v_n-v_*\|  \| x-w_n\| - \| w_n-u\|,
\end{align*}
which yields the existence of $n_0$ such that $ \left\langle v_n, x-w_n \right\rangle\geq \alpha/2$ for $n\geq n_0$. As a consequence,  from \eqref{eq:cos-to-zero},  $r_n\to +\infty$ and the boundedness of $K$ and $w_n$, we derive the existence of $n_0\in \NN$ such that, for every $x\in K$ and $n\geq n_0$, 
$$  \| x-w_n\|  \cos\left( \frac{ \pi /2}{1+r_n^{-1}} \right) \leq \left\langle v_n, x-w_n \right\rangle \leq r_n .$$
 That is, for $n\geq n_0$ we have
$$K\subset {\rm Cone}\left(r_n, \frac{ \pi }{1+r_n^{-1}}, {v}_n,w_n\right),  $$
which by construction of the cone yields  
$$\{x-q_n: x\in K\} \subset {\rm Cone}\left(r_n, \frac{ \pi }{1+r_n^{-1}}, {v}_n,w_n-q_n\right) \subset \{ x: h(x) \leq  h(w_n-q_n) \}. $$
From this we conclude that 
$K \subset \{ x: h(x-q_n) \leq  h(w_n-q_n) \}$ for $n$ large enough. Hence,  (iii) follows from (ii). 
\end{proof}

    \begin{proof}[Proof of \Cref{lemma:do-not-scape}]
    Let $K$ and $K'$ be a compact sets such that   $Q(K)\geq \varepsilon^{\frac{1}{3}} $ and $P(K')\geq \frac{1-\varepsilon^{\frac{1}{2}}}{1-\varepsilon} \in (0,1)$. Since $(\nabla^c f_n)\# Q=  (1-\varepsilon)P+\varepsilon \nu_n$,  then   $$ Q((\nabla^c f_n)^{-1}(K')) \geq (1-\varepsilon) P(K')\geq  1-\varepsilon^{\frac{1}{2}} , $$
   which implies that
   \begin{align*}
       Q( K\cap (\nabla^c f_n)^{-1}(K')) &\geq   1-  Q( \R^d\setminus K) - Q( \R^d\setminus (\nabla^c f_n)^{-1}(K') )\\
       &=  Q( K) + Q((\nabla^c f_n)^{-1}(K') ) -1 \geq \varepsilon^{\frac{1}{3}} - \varepsilon^{\frac{1}{2}}>0. 
   \end{align*}
From this, it follows that $K\cap (\nabla^c f_n)^{-1}(K')$ is nonempty, so that there exists $x_n\in K$ with $\partial^c f_n(x_n)\in K'$ and the result follows.  
\end{proof}

\bibliographystyle{plainnat}
\bibliography{biblio}

\end{document}